\documentclass{amsart}
\pdfoutput=1
\usepackage{graphicx}
\usepackage{subfigure}
\usepackage{url}
\usepackage{amsmath}
\usepackage{amssymb}
\usepackage{amsopn}
\usepackage{amssymb}
\usepackage{amstext}
\usepackage{amsthm}
\usepackage{float}
\usepackage{setspace}
\usepackage{color}
\usepackage[T1]{fontenc}
\usepackage{algorithm}
\usepackage{algpseudocode}

\def \fig_size{ 3.5 }

\newtheorem{lemma}{Lemma}

\newcommand{\diag}{\operatorname{diag}}
\newcommand{\pattern}{\operatorname{Pattern}}

\newcommand{\Hinf}{\mathcal H_\infty}

\let\oldsum\sum
\renewcommand{\sum}{\displaystyle\oldsum}
\let\oldmin\min
\renewcommand{\min}{\displaystyle\oldmin}
\let\oldmax\max
\renewcommand{\max}{\displaystyle\oldmax}
\let\oldinf\inf
\renewcommand{\inf}{\displaystyle\oldinf}
\let\oldsup\sup
\renewcommand{\sup}{\displaystyle\oldsup}

\begin{document}

\title{How To Tame Your Sparsity Constraints}

\author{J. A. Lopez}
\thanks{The author is with the Department of Electrical \& Computer Engineering, Northeastern University, MA 02115, USA. E-mail: {\tt lopez.jo@husky.neu.edu}. This work was supported in part by the NSF Integrative Graduate Education and Research Traineeship (IGERT) program under grant DGE-0654176.}

\keywords{robust control, sparse controller}

\begin{abstract}
We show that designing sparse $H_\infty$ controllers, in a discrete (LTI) setting, is easy when the controller is assumed to be an FIR filter. In this case, the problem reduces to a static output feedback problem with equality constraints. We show how to obtain an initial guess, for the controller, and then provide a simple algorithm that alternates between two (convex) feasibility programs until converging, when the problem is feasible, to a suboptimal $H_\infty$ controller that is automatically stable. As FIR filters contain the information of their impulse response in their coefficients, it is easy to see that our results provide a path of least resistance to designing sparse robust controllers for continuous-time plants, via system identification methods.
\end{abstract}

\date{May 30, 2015}

\maketitle

\section{Introduction}
Recently there has been much interest in designing controllers that satisfy so-called ``sparsity constraints'' -- a catch-all term for constraints that arise from information flow restrictions \cite{rotkowitz_lall_2006,wang_lopez_sznaier,lopez_wang_sznaier,veillette_et_al_1992}, by the control-loop topology \cite{aircraft_control_text,oliveira_geromel_bernussou_2000}, etc. The main difficulty in sparse controller design comes from the obscure relationship between sparsity constraints, imposed on the pattern of the closed-loop transfer matrix, and the closed-loop state-space description; which is where we prefer to work since modern control theory relies heavily on tools from linear algebra. As with non-sparse robust control design, there seem to be two principal roads to controller synthesis: the Youla parametrization and the linear matrix inequality (LMI) solutions. Some sparse control researchers have focused on determining theoretical conditions under which the Youla parametrization solution remains convex under sparsity constraints \cite{rotkowitz_lall_2006}; while others have made progress by adapting existing LMI tools so that they can extract sparse controllers by assuming an easier-to-work-with structure on the controller or in the LMI formulation \cite{oliveira_geromel_bernussou_2000,polyak_et_al_2013,palacios_et_al_2012}. Our approach also imposes a structure on the problem: we assume the controller is an FIR filter (and therefore also assume the plant is strongly stabilizable) but we do not require the sparsity pattern to be in block form nor do we impose a structure on the LMI factors. Moreover, our method uses an alternating scheme that culminates with a controller and a certificate of stability and robust performance. Perhaps the closest work to ours is \cite{lin_fardad_jovanovic}, where the authors also use an alternating scheme to design sparse controllers for the $H_2$ (static) state-feedback case. Our approach can be considered a generalization of their work since it includes static state-feedback controller design as a special case and can easily incorporate $H_2$ constraints, if desired. However, unlike \cite{lin_fardad_jovanovic}, we will not search for an optimal solution since there is little to gain from achieving the lowest $H_\infty$ norm \cite{mario_text}.

There are many advantages to using an FIR controller. We will show that by imposing this form on the controller, the problem of synthesizing (sparse) suboptimal $H_\infty$ controllers reduces to a static output-feedback problem, for an augmented system, with the sparsity constraints entering the problem through simple equality constraints. Although imposing an FIR structure also assumes the plant is strongly stabilizable, this assumption is not overly restrictive since it is not practical to implement unstable controllers, anyway \cite{cheng_cao_sun_2011,petersen_2009,chou_wu_leu_2003}. Apart from this assumption, we only assume that the plant is detectable and that the control inputs do not directly feed-through to the output. Thus, our method only has mild restrictions and we will demonstrate that the resulting controllers attain excellent bounds on the $H_\infty$ norm. The controllers are ultimately obtained using a simple algorithm that alternates between two (convex) feasibility programs until convergence, when the problem is feasible -- while we do not pursue convergence guarantees, we take comfort in knowing that our algorithm is simpler than other alternating schemes (ADMM and expectation-maximization) for which convergence proofs are available \cite{boyd_admm_2011,bishop_text}. Furthermore, other alternating schemes like ``D-K iteration'' have had success in the past.

In the following sections we will state the problem, present our main contribution, and demonstrate its use with examples.

\subsection{Preliminaries} 
Our notation is standard from the robust control literature \cite{mario_text, zhou_text, dullerud_paganini_text}. We consider the standard robust control loop shown in Figure \ref{fig:general_loop_structure}: where $G$ represents the discrete-time plant, $w$ represents the disturbance inputs, $u$ the control inputs, $z$ the regulated outputs, and $y$ the outputs available to the discrete-time controller, $K$ \cite{mario_text}. 

\begin{figure}[H]
\centering
\def\svgwidth{200pt}
\resizebox{\fig_size in}{!}{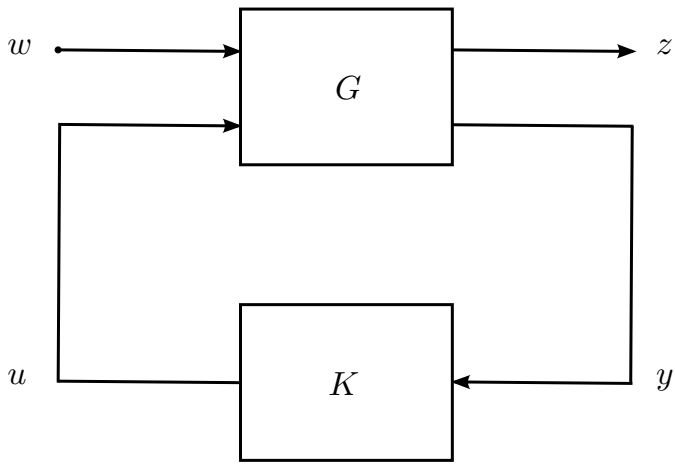}
\caption{Closed-loop system}
\label{fig:general_loop_structure}
\end{figure}

The state-space description is given by:
\begin{align}
x_{k+1} &= Ax_k + B_1 w_k + B_2 u_k\\
z_k &= C_1 x_k + D_{11} w_k + D_{12} u_k\notag\\
y_k &= C_2 x_k + D_{21} w_k \notag
\label{eq:system_equations}
\end{align}

For this system, we would like to minimize the $\Hinf$ norm of the transfer function from the disturbance inputs to the regulated outputs:
\begin{equation}
T_{zw} = \left[
\begin{array}{c|c}
A_{CL} & B_{CL}\\
\hline
C_{CL} & D_{CL}
\end{array}
\right]
\label{eq:tzw}
\end{equation}
where $(A_{CL}, B_{CL}, C_{CL}, D_{CL})$ is the system with an output feedback controller $u = Ky$ in the loop. 
Formally, we consider Problem 1 below.\\\\
\textbf{Problem 1.} Given $\mu>0$ and a binary pattern matrix, $S$, find a stabilizing controller $K\in S$ such that $\|T_{zw}\|_\infty^2<\mu$.\\\\
\textbf{Assumptions:}
The assumptions we make are that the pair $(A,B_2)$ is strongly stabilizable, $(A,C_2)$ is detectable, and that $D_{22}=0$ \cite{gahinet_apkarian}. \\

For our solution, we turn to a standard result in the robust control literature \cite{oliveira_geromel_bernussou_2002,bu_sznaier,gahinet_apkarian,dullerud_paganini_text}.

\begin{lemma} The inequality $\|T_{zw}\|_\infty^2<\mu$ holds if, and only if, there is a symmetric matrix $P$ such that
\begin{equation}
\left[
\begin{array}{cccc}
P & AP & B & 0\\
PA^T & P & 0 & PC^T\\
B^T & 0 & I & D^T\\
0 & CP & D & \mu I
\end{array}
\right]>0
\label{eq:bounded_real_lemma}
\end{equation} is feasible.
\label{lemma:bounded_real_lemma}
\end{lemma}
Next, we include a review on FIR realizations from \cite{rotstein_sideris_1993, rotstein_sideris_1994}. For this review, we will temporarily abuse the notation by also using $z$ for the traditional $z$-transform variable. If $Q(z)$ is an, $l$-input, FIR filter with $Q(z)=\sum_{i=0}^{n-1}Q_iz^{-i}$, then one state-space realization is: 
\begin{equation}
Q = \left[
\begin{array}{c|c}
A_q^{n-1} & E_{n-1}^{n-1}\\
\hline
[Q_{n-1},\dots,Q_1] & Q_0
\end{array}
\right] 
\end{equation}
with 
\begin{equation}
A_q^{n} = \left[
\begin{array}{ccccc}
0 & I_l & 0 & \dots & 0\\
0 & 0 & I_l & \dots & 0\\
\vdots & \vdots & \vdots & \dots & I_l\\
0 & 0 & 0 & \dots & 0
\end{array}
\right]\in\mathbb{R}^{nl\times nl}
\end{equation}
\begin{equation}
E_1^{n} = \left[
\begin{array}{c}
I_l\\
0\\
\vdots\\
0
\end{array}
\right],\dots,
E_n^{n} = \left[
\begin{array}{c}
0\\
0\\
\vdots\\
I_l
\end{array}
\right]
\in\mathbb{R}^{nl\times l}
\end{equation}
For example, if
\begin{equation}
\begin{array}{ccc}
Q(z)=[1 + 2z^{-1} + 3z^{-2} & 1 + 2z^{-1} + 3z^{-2}  & 1 + 2z^{-1} + 3z^{-2}] 
\end{array}
\end{equation}
then we will assume its state-space description has the form:
\begin{align}
&A_q=\left[
\begin{array}{cc}
0 & I_3\\
0 & 0 
\end{array}
\right],\;
B_q=\left[
\begin{array}{c}
0 \\
I_3
\end{array}
\right],\;
C_q=\left[
\begin{array}{cc}
Q_2 & Q_1
\end{array}
\right],D_q = Q_0,\text{ with }\\
&Q_0 = \left[
\begin{array}{ccc}
1 & 1 & 1
\end{array}
\right],\;Q_1 = \left[
\begin{array}{ccc}
2 & 2 & 2
\end{array}
\right],\text{ and }Q_2 = \left[
\begin{array}{ccc}
3 & 3 & 3
\end{array}
\right].\notag
\end{align}
Thus, besides being automatically stable, FIR filters have the satisfying property of clearly identifying entries in its transfer function with entries in its state-space description. (Clearly, $Q(j,k)=0$ if, and only if, $Q_i(j,k)=0$ for all $i$.) This obviates any difficulty that comes with trying to associate entries in a general transfer matrix (i.e., one not required to be an FIR filter) with its state-space description via the transformation $G(z) = C(zI-A)^{-1}B+D$ which has an unfriendly inversion in the way. Furthermore, for a fixed $n$, $A_q$ and $B_q$ are also fixed. We will use this property, in the next section, where we show that Problem 1 reduces to a static output feedback problem, for an augmented system, when we require that:
\begin{equation}
K = \left[
\begin{array}{c|c}
A_q & B_q\\
\hline
C_q & D_q
\end{array}
\right].
\label{eq:K_fir}
\end{equation}

\section{Main result}
In this section, we prove the following claim.
\begin{lemma}
Problem 1 reduces to a static output feedback problem for the augmented system $(A_o,B_o,C_o,D_o)$ where 
\begin{align}
A_o &= \left[
\begin{array}{cc}
A & 0\\
B_qC_2 & A_q
\end{array}
\right],\quad B_o = \left[
\begin{array}{c}
B_1\\
B_qD_{21}
\end{array}
\right], \\
C_o &= \left[
\begin{array}{cc}
C_1 & 0
\end{array}
\right],\quad D_o=D_{11}.\notag
\label{eq:Ao_system}
\end{align}
\end{lemma}
\begin{proof}
With $K$ as in (\ref{eq:K_fir}), the closed-loop system is given by 
\begin{align}
A_{CL} &= \left[
\begin{array}{cc}
A+B_2D_qC_2 & B_2C_q\\
B_qC_2 & A_q
\end{array}
\right],\; B_{CL} = \left[
\begin{array}{c}
B_1+B_2D_qD_{21}\\
B_qD_{21}
\end{array}
\right]\\
C_{CL} &= \left[
\begin{array}{cc}
C_1+D_{12}D_qC_2 & D_{12}C_q
\end{array}
\right], \; D_{CL}=D_{11}+D_{12}D_qD_{21}.
\label{eq:closed_loop_system}
\end{align} Following standard approaches \cite{mario_text,gahinet_apkarian}, we can parametrize the closed-loop system as:
\begin{align}
A_{CL}  &= A_o + \tilde{B}K_o\tilde{C},\; B_{CL} = B_o + \tilde{B}K_o\tilde{D}_{21}\\
C_{CL} &= C_o + \tilde{D}_{12} K_o \tilde{C},\; D_{CL} = D_o + \tilde{D}_{12}K_o \tilde{D}_{21}
\label{eq:parametrized_closed_loop_system}
\end{align}
with
\begin{align}
&\tilde{B} = \left[
\begin{array}{c}
B_2\\
0
\end{array}
\right],\; \tilde{C} = \left[
\begin{array}{cc}
0 & I\\
C_2 & 0
\end{array}
\right],\;\tilde{D}_{12}=D_{12},\\
&\tilde{D}_{21} = \left[\begin{array}{c}
0\\
D_{21}
\end{array}\right],\text{ and }K_o = \left[\begin{array}{cc} C_q & D_q \end{array}\right].
\label{eq:tilde_system}
\end{align}
Thus, the stability of the closed-loop system is determined by $A_o + \tilde{B}K_o\tilde{C}$; which the reader can recognize as the traditional output-feedback form.
\end{proof}

Substituting the parametrized system into Lemma \ref{lemma:bounded_real_lemma} gives the required condition:

\begin{align}
F(P,K_o) \triangleq \left[
\begin{array}{cccc}
P & (A_o+\tilde{B}K_o\tilde{C})P & B_o+\tilde{B}K_o\tilde{D}_{21} & 0\\
P(A_o+\tilde{B}K_o\tilde{C})^T & P & 0 & P(C_o+\tilde{D}_{12} K_o \tilde{C})^T\\
(B_o+\tilde{B}K_o\tilde{D}_{21})^T & 0 & I & (D_o+\tilde{D}_{12}K_o \tilde{D}_{21})^T\\
0 & (C_o+\tilde{D}_{12} K_o \tilde{C})P & D_o+\tilde{D}_{12}K_o \tilde{D}_{21} & \mu I
\end{array}
\right] >0
\label{eq:full_condition}
\end{align}

At this point, traditional approaches would deal with the non-convexity of this condition by lifting the non-convex terms, solving Lemma \ref{lemma:bounded_real_lemma}, and extracting $K_o$ from the lifted variable using an inversion. The inversion is problematic when enforcing sparsity constraints and usually researchers require $P$ to be diagonal so that a sparsity pattern can more easily be imposed on the controller. However, there may not be such a $P$. Here, we take an iterative approach. Define the matrix, $F_0$, as: 

\begin{align}
F_0(P,K_o) \triangleq
\left[
\begin{array}{cccc}
P & A_oP+\tilde{B}K_o\tilde{C} & B_o+\tilde{B}K_o\tilde{D}_{21} & 0\\
PA_o^T+(\tilde{B}K_o\tilde{C})^T & P & 0 & PC_o^T+(\tilde{D}_{12} K_o \tilde{C})^T\\
(B_o+\tilde{B}K_o\tilde{D}_{21})^T & 0 & I & (D_o+\tilde{D}_{12}K_o \tilde{D}_{21})^T\\
0 & C_oP+\tilde{D}_{12} K_o \tilde{C} & D_o+\tilde{D}_{12}K_o \tilde{D}_{21} & \mu I
\end{array}
\right]>0 
\end{align}

which simply removes $P$ from the non-convex terms. Define the two feasibility programs, $\mathbb{P}_0$ and $\mathbb{P}$, as:
\begin{align}
&\mathbb{P}_0(P,K_o) \mapsto \left\lbrace 
\begin{array}{c}
\min_{P,K_o} 0\\
\text{ subject to }
Q_i\in S,\;\forall i\\
F_0(P,K_o)>0
\end{array}
\right.\;\text{ and }\\
&\mathbb{P}(P,K_o) \mapsto \left\lbrace 
\begin{array}{c}
\min_{P,K_o} 0\\
\text{ subject to }
Q_i\in S,\;\forall i\\
F(P,K_o)>0
\end{array}
\right.
\end{align}
where $Q_i\in S$ is enforced using equality constraints determined by the zeros in $S$. 

We now define Algorithm \ref{alg:main_algorithm}, to synthesize $\mu$-optimal controllers, using SDP solvers. Below, the notation $\mathbb{P}(P,K_o)^k$ means $k$-iterations of the feasibility program $\mathbb{P}$ (which is easily specified using the maximum-iterations input of the solver) and the bolded terms identify the decision variables in the program while the other terms are fixed, from the previous iteration.

\begin{algorithm}
\caption{Sparse $\Hinf$ Control Synthesis}\label{alg:main_algorithm}
\renewcommand{\algorithmicrequire}{\textbf{Input:}}
\renewcommand{\algorithmicensure}{\textbf{Output:}}
\begin{algorithmic}[1]
\Require $\mu>0$, FIR order $n$, sparsity pattern $S$
\Ensure $(P,K)$
\State $K_o \gets \mathbb{P}_0({\bf P, K_o})^{k_0}$ 
\While{$F(P,K_o)\le 0$}
\State $P \gets \mathbb{P}({\bf P},K_o)^{k_1}$
\If{$F(P,K_o)>0$}
   \State \textbf{return} $(P,K_o)$
\Else   
   \State $K_o \gets \mathbb{P}(P, {\bf K_o})^{k_2}$
\EndIf   
\EndWhile
\end{algorithmic}
\end{algorithm}

In practice, Algorithm \ref{alg:main_algorithm} converged faster when the SDP solver was allowed more iterations when searching for $P$ (i.e. when $k_1>k_2$). We also found that changing $\mu$ affected the speed of convergence. Setting $k=[2,5,2]$ and then increasing $k_0$ with $\mu$ set larger than the squared ($\mathcal{H}_\infty$) norm of the closed-loop system with the full controller is a good place to start\footnote{For convenience, full $H_\infty$ controller can be obtained using the MATLAB function {\tt hinfsyn}.}.

\subsection{Pattern Discovery}
In \cite{lin_fardad_jovanovic}, Lin et al. identified sparsity patterns using ``sparsity promoting penalty functions''. Since $\mathbb{P}_0$ is a relaxation of $\mathbb{P}$, it can be used to discover sparse patterns by simply adding an objective to $\mathbb{P}_0$ that minimizes, for example, the re-weighted $\ell_1$-norm of $K_o$ \cite{candes_wakin_boyd_arxiv}. 
\begin{algorithm}
\caption{Pattern Discovery}\label{alg:sparsest_gain_algorithm}
\renewcommand{\algorithmicrequire}{\textbf{Input:}}
\renewcommand{\algorithmicensure}{\textbf{Output:}}
\begin{algorithmic}[1]
\Require $\mu>0$, FIR order $n$, max iterations $N$, re-weighting constant $\epsilon>0$
\Ensure $\pattern(K_o)$
\State $W_0 \gets {\bf{1}}$ 
\While{$k<N$}
\State ${K_o}_k \gets \left\lbrace 
\begin{array}{c}
\min_{P,K_o} \| W_{k-1}\cdot {K_o}_{k-1}\|_1\\
\text{ subject to }
F_0(P,K_o)>0
\end{array}
\right.$
\State ${K_o}_{k-1} \gets {K_o}_k$
\State $W_k(i,j) \gets 1/\left(|{K_o}_k(i,j)+\epsilon|\right)$
\State $k \gets k + 1$
\EndWhile
\State \textbf{return} $\pattern(K_o)$
\end{algorithmic}
\end{algorithm}

\section{Examples}
The examples below were solved using MATLAB/CVX/SeDuMi/SDPT3 \cite{matlab,cvx,sedumi,sdpt3}.  \\\\
\textbf{Example 1}
This example was used to demonstrate distributed $\mathcal{H}_\infty$ synthesis in \cite{veillette_et_al_1992} and distributed $\mathcal{H}_2$ synthesis in \cite{oliveira_geromel_bernussou_2000}. The goal is to design a diagonal controller for the (unstable) continuous-time plant below.

\begin{align}
A &= \left[
\begin{array}{cccc}
-2 & 1 & 1 & 1\\
3 & 0 & 0 & 2\\
-1 & 0 & -2 & -3\\
-2 & -1 & 2 & -1
\end{array}
\right],\; B_1 = \left[
\begin{array}{ccc}
1 & 0 & 0\\
0 & 0 & 0\\
1 & 0 & 0\\
0 & 0 & 0
\end{array}
\right]\\
 B_2 &= \left[
\begin{array}{cc}
0 & 0\\
1 & 0\\
0 & 0\\
0 & 1
\end{array}
\right],\;
C_1 = \left[
\begin{array}{cccc}
1 & 0 & -1 & 0\\
0 & 0 & 0 & 0\\
0 & 0 & 0 & 0
\end{array}
\right]\notag\\
C_2 &= \left[
\begin{array}{cccc}
1 & 0 & 0 & 0\\
0 & 0 & 1 & 0
\end{array}
\right],\;
D_{11} = 0_3, D_{12} = \left[
\begin{array}{cc}
0 & 0\\
1 & 0\\
0 & 1
\end{array}
\right],\notag\\
D_{21} &= \left[
\begin{array}{ccc}
0 & 1 & 0\\
0 & 0 & 1
\end{array}
\right],\;D_{22}=0_2\notag
\end{align}
After discretizing using a zero-order hold with a 0.1s sampling period, we redefine $A,B_1,$ and $B_2$ as:
\begin{align}
A &= \left[
\begin{array}{cccc}
0.8189 & 0.08627 & 0.09004 & 0.08133\\
0.2524 & 1.003 & 0.03134 & 0.2004\\
-0.05449 & 0.01017 & 0.7901 & -0.258\\
-0.1918 & -0.1034 & 0.1602 & 0.8604
\end{array}
\right],\\
B_1 &= \left[
\begin{array}{ccc}
0.09531 & 0 & 0\\
0.01447 & 0 & 0\\
0.08618 & 0 & 0\\
-0.001083 & 0 & 0
\end{array}
\right]\\
 B_2 &= \left[
\begin{array}{cc}
0.004532 & 0.004367\\
0.1001 & 0.01005\\
0.0003383 & -0.01361\\
-0.005126 & 0.09363
\end{array}
\right]
\end{align}
Running Algorithm \ref{alg:main_algorithm}, using $\mu=24$ and $k=[10,5,2]$, we obtain the following static controller ($P$ is not shown to conserve space):
\begin{align}
K_0 &= \left[
\begin{array}{cc}
-1.2775 & 0\\
0 & -0.5685
\end{array}
\right]
\end{align}
which results in $\|T_{zw}\|_\infty=1.85$ for the continuous-time system. Using the same $k$, and $\mu=9$, we obtain the following first-order controller:
\begin{align}
K_1 &= \left[
\begin{array}{cc}
\frac{-0.6543 z - 0.5344}{z} & 0\\
0 & \frac{-0.1993 z - 0.2237}{z}
\end{array}
\right]
\end{align}
which results in $\|T_{zw}\|_\infty=1.9043$ for the discrete-time system.

The continuous-time equivalent was identified (individually), using the MATLAB function {\tt tfest(id\_data,1,0)}, to obtain
\begin{align}
K_1 &= \left[
\begin{array}{cc}
\frac{-8.469}{s + 7.382} & 0\\
0 & \frac{-4.167}{s + 10.24}
\end{array}
\right]
\end{align}
which results in $\|T_{zw}\|_\infty=1.9517$ for the continuous-time system.
\begin{align}
K_2 &= \left[
\begin{array}{cc}
\frac{-0.6377 z^2 - 0.2514 z - 0.3369}{z^2} & 0\\
0 & \frac{-0.228 z^2 - 0.1406 z - 0.2332}{z^2}
\end{array}
\right]
\end{align}
was obtained using $\mu=9$ and $k=[12,5,2]$ -- this controller results in $\|T_{zw}\|_\infty=1.9795$. 

The continuous-time equivalent is given by: 
\begin{align}
K_2 &= \left[
\begin{array}{cc}
\frac{-189.7}{s^2 + 40.21 s + 157.8} & 0\\
0 & \frac{-6.11e04}{s^2 + 565.7 s + 1.254e05}
\end{array}
\right]
\end{align}
and results in $\|T_{zw}\|_\infty=1.9711$. All three controllers have slightly better closed-loop $\mathcal{H}_\infty$ performance than the ones obtained in \cite{veillette_et_al_1992}, the best of which result in $\|T_{zw}\|_\infty=1.995$.\\\\
\textbf{Example 2} 
The following continuous-time example is from \cite{rotkowitz_lall_2006}.
\begin{align}
G(s) & = \left[
\begin{array}{ccccc}
\frac{1}{s+1} & 0 & 0 & 0 & 0\\
\frac{1}{s+1} & \frac{1}{s-1} & 0 & 0 & 0\\
\frac{1}{s+1} & \frac{1}{s-1} & \frac{1}{s+1} & 0 & 0\\
\frac{1}{s+1} & \frac{1}{s-1} & \frac{1}{s+1} & \frac{1}{s+1} & 0\\
\frac{1}{s+1} & \frac{1}{s-1} & \frac{1}{s+1} & \frac{1}{s+1} & \frac{1}{s-1}
\end{array}
\right]
\end{align}
For this model, the following realization is used:
\begin{align}
A &= \left[
\begin{array}{ccccc}
-1 & 0 & 0 & 0 & 0\\
0 & 1 & 0 & 0 & 0\\
0 & 0 & -1 & 0 & 0\\
0 & 0 & 0 & -1 & 0\\
0 & 0 & 0 & 0 & 1
\end{array}
\right]\\ B_1 &= \left[
\begin{array}{ccccc}
0.2 & 0 & 0 & 0 & 0\\
0 & 0.2 & 0 & 0 & 0\\
0 & 0 & 0.2 & 0 & 0\\
0 & 0 & 0 & 0.2 & 0\\
0 & 0 & 0 & 0 & 0.1
\end{array}
\right]\\
B_2 &= \left[
\begin{array}{ccccc}
2 & 0 & 0 & 0 & 0\\
0 & 2 & 0 & 0 & 0\\
0 & 0 & 2 & 0 & 0\\
0 & 0 & 0 & 2 & 0\\
0 & 0 & 0 & 0 & 1
\end{array}
\right]\\
C_1 &= \left[
\begin{array}{ccccc}
0.05 & 0 & 0 & 0 & 0\\
0.05 & 0.05 & 0 & 0 & 0\\
0.05 & 0.05 & 0.05 & 0 & 0\\
0.05 & 0.05 & 0.05 & 0.05 & 0\\
0.05 & 0.05 & 0.05 & 0.05 & 0.1
\end{array}
\right]\\ 
C_2 &= \left[
\begin{array}{ccccc}
0.5 & 0 & 0 & 0 & 0\\
0.5 & 0.5 & 0 & 0 & 0\\
0.5 & 0.5 & 0.5 & 0 & 0\\
0.5 & 0.5 & 0.5 & 0.5 & 0\\
0.5 & 0.5 & 0.5 & 0.5 & 1
\end{array}
\right]\notag\\
D_{11} &= D_{12} = D_{22} = D_{22} = 0_5.\notag
\end{align}
After discretizing, using a sampling period of 0.05 s, redefine $A,\,B_1,$ and $B_2$ as:
\begin{align}
A &= \left[
\begin{array}{ccccc}
0.9512 & 0 & 0 & 0 & 0\\
0 & 1.051 & 0 & 0 & 0\\
0 & 0 & 0.9512 & 0 & 0\\
0 & 0 & 0 & 0.9512 & 0\\
0 & 0 & 0 & 0 & 1.051
\end{array}
\right]\\
B_1 &= \left[
\begin{array}{ccccc}
0.009754 & 0 & 0 & 0 & 0\\
0 & 0.01025 & 0 & 0 & 0\\
0 & 0 & 0.009754 & 0 & 0\\
0 & 0 & 0 & 0.009754 & 0\\
0 & 0 & 0 & 0 & 0.005127
\end{array}
\right], B_2 &= 10B_1
\end{align}
As in \cite{rotkowitz_lall_2006}, we would also like to find a first-order controller that satisfies the sparsity constraint:
\begin{equation}
S = \left[
\begin{array}{ccccc}
0 & 0 & 0 & 0 & 0\\
0 & 1 & 0 & 0 & 0\\
0 & 0 & 0 & 0 & 0\\
0 & 0 & 0 & 0 & 0\\
0 & 0 & 0 & 0 & 1
\end{array}
\right]
\end{equation} Setting $\mu=0.1$ and $k=[2,5,2]$, Algorithm \ref{alg:main_algorithm} obtains the following controller:
\begin{align}
&K_1 = \left[
\begin{array}{ccccc}
0 & 0 & 0 & 0 & 0\\
0 & \frac{-9.265 z - 5.227}{z} & 0 & 0 & 0\\
0 & 0 & 0 & 0 & 0\\
0 & 0 & 0 & 0 & 0\\
0 & 0 & 0 & 0 & \frac{-12.68 z - 15.84}{z}
\end{array}
\right]
\end{align}
which results in $\|T_{zw}\|_\infty=0.0163$. 

The continuous-time equivalent was identified, using the MATLAB function {\tt tfest(id\_data,1,0)}, and re-assembled to obtain:
\begin{align}
&K_1 = \left[
\begin{array}{ccccc}
0 & 0 & 0 & 0 & 0\\
0 & \frac{-120}{s + 8.501} & 0 & 0 & 0\\
0 & 0 & 0 & 0 & 0\\
0 & 0 & 0 & 0 & 0\\
0 & 0 & 0 & 0 & \frac{-616.7}{s + 22.48}
\end{array}
\right]
\end{align} which results in the same closed-loop cost. Additionally, a static controller can be obtained (using the same settings):
\begin{align}
&K_0 = \left[
\begin{array}{ccccc}
0 & 0 & 0 & 0 & 0\\
0 & -9.1093 & 0 & 0 & 0\\
0 & 0 & 0 & 0 & 0\\
0 & 0 & 0 & 0 & 0\\
0 & 0 & 0 & 0 & -12.372
\end{array}
\right]
\end{align}
which results in the same cost, $\|T_{zw}\|_\infty=0.0165$, for both systems.\\\\
\textbf{Example 3}
A limitation of our approach is that, presently, SDP solvers can only handle medium-sized problems. For instance, the examples in \cite{lin_fardad_jovanovic} were too large to solve on a 2-core computer with 4GB of memory. On the other hand, we were able to identify sparser patterns and synthesize sparser controllers than their approach. This example is a smaller version of the mass-spring example in \cite{lin_fardad_jovanovic}. The plant model is a classical mass-spring system with $N_m=8$ unit masses and spring constants.

\begin{align}
A&=\left[
\begin{array}{cc}
0 & I\\
-T & 0
\end{array}
\right],\; B_1=B_2=\left[
\begin{array}{c}
0 \\
I
\end{array}
\right],\\
C_1&=C_2=Q^{1/2},\; D_{11}=D_{21}=\left[
\begin{array}{c}
0 \\
0
\end{array}
\right],\\
D_{12}&=D_{22}=\left[
\begin{array}{c}
R^{1/2} \\
0
\end{array}
\right]
\end{align}
with $T$ an $N_m\times N_m$ tridiagonal Toeplitz matrix with $-2$ on the main diagonal and $1$ on the sub- and super diagonals. The $I$ and $0$ terms are $N_m\times N_m$ identity- and zero matrices, respectively. $R=10I$ and $Q=I_{2N_m}$. After discretizing, using a sampling period of 0.5 s, the $A,B_1,$ and $B_2$ matrices were redefined as in the previous examples. 

A controller for this system is partitioned as $K_0=[K_p, K_v]$ corresponding to the position- and velocity states. Applying Algorithm \ref{alg:sparsest_gain_algorithm}, with $\mu=80$, $n=2$, $N=40$, revealed diagonal patterns for $K_p$ and $K_v$, which agrees with \cite{lin_fardad_jovanovic}. However, the entries corresponding to $K_p$ were closer to zero and  several times smaller than the $K_v$ terms. Thus, we hypothesized that it is possible to design controllers with the pattern $S=[0,I]$ -- this agrees with our intuition that it should be possible to control this system using velocity feedback. 

 Setting $\mu=83$ and $k=[10,5,2]$, Algorithm \ref{alg:main_algorithm} obtained the following controller:
 \begin{align}
 K_p &= 0_{N_m}\\
K_v &= \diag\left(\left[
\begin{array}{c}
-0.7491\\
-0.706\\
-0.7255\\
-0.7316\\
-0.7316\\
-0.7255\\
-0.706\\
-0.7529
\end{array}
\right]\right)
 \end{align}
at a cost of $\|T_{zw}\|_\infty=8.2909$. 

\section{Conclusions}
We presented a very simple method for solving a very difficult problem that seems to work well for strongly stabilizable plants. Broadly speaking, the idea behind the solution is in line with other approaches to attacking non-convex problems: by breaking it up into smaller problems or steps\footnote{Indeed, the standard LMI approach to robust control synthesis is a great example of how to break-up a non-convex problem into separate convex problems.}. The method is aided by the fact that $\mathbb{P}_0$ tends to provide very good initial guesses and probably deserves more attention on its own. 

The method is flexible and, as in Algorithm \ref{alg:sparsest_gain_algorithm}, the feasibility programs can include other objectives. The flexibility can be used to optimize additional objectives and it can be used to enhance convergence for very difficult cases. For example, in Example 3, it is possible to synthesize a still sparser $K_v$ if we minimize $\|P\|_2+\|Q_i\|_2$ in $\mathbb{P}_0$. We left this out for clarity, but it could be researched further.

\nocite{*}
{\small
\bibliographystyle{plain}
\bibliography{references}
}


\end{document}